\documentclass[12pt]{article}
\usepackage{amssymb}
\usepackage{amsmath}
\usepackage{amsthm}
\usepackage{graphicx}

\usepackage[numbers,sort&compress]{natbib}

\newtheorem{theorem}{Theorem}[section]
\newtheorem{definition}[theorem]{Definition}

\newtheorem{corollary}[theorem]{Corollary}
\newtheorem{example}[theorem]{Example}
\newtheorem{lemma}[theorem]{Lemma}

\renewenvironment{proof}{
\noindent{\bf Proof.}\rm} {\mbox{}\hfill\rule{0.5em}{0.809em}\par}

\newcommand{\na}{GDN algebra}

\newcommand{\cca}{special GDN-Poisson admissible algebra}

\newcommand{\npa}{GDN-Poisson algebra}

\newcommand{\npas}{GDN-Poisson algebras}

\newcommand{\npt}{GDN-Poisson tableau}

\newcommand{\nptx}{GDN-Poisson tableaux}

\newcommand{\nt}{GDN tableau}

\newcommand{\ma}{$\mathcal{A}$ }

\begin{document}

\title{On free Gelfand-Dorfman-Novikov-Poisson algebras and  a PBW theorem\footnote{Supported by the NNSF of China (11171118, 11571121).}}

\author{
L. A. Bokut\footnote {Supported by Russian Science Foundation (project
14-21-00065).} \\
{\small \ School of Mathematical Sciences, South China Normal
University}\\
{\small Guangzhou 510631, P. R. China}\\
{\small Sobolev Institute of mathematics, Novosibirsk, 630090}\\
 {\small Novosibirsk State University, Novosibirsk 630090, Russia}\\
{\small  bokut@math.nsc.ru}\\
\\
 Yuqun
Chen\footnote {Corresponding author.} \  and Zerui Zhang\footnote{Supported by the Innovation Project of Graduate School of South China Normal University.
}\\
{\small \ School of Mathematical Sciences, South China Normal
University}\\
{\small Guangzhou 510631, P. R. China}\\
{\small yqchen@scnu.edu.cn}\\
{\small 295841340@qq.com}}

\date{}

\maketitle

\noindent{\bf Abstract:} In 1997, X. Xu \cite{Xiaoping Xu 2, Xiaoping Xu Poisson} invented a concept of Novikov-Poisson algebras
(we call them Gelfand-Dorfman-Novikov-Poisson  (GDN-Poisson) algebras).
We construct a linear basis of a free GDN-Poisson algebra. We define
a notion of a special GDN-Poisson admissible algebra, based on X. Xu's definition
and an S.I. Gelfand's observation (see \cite{Gelfand}). It is a differential algebra with two commutative
associative products and some extra identities. We prove that any GDN-Poisson algebra
is embeddable into its universal enveloping special GDN-Poisson admissible algebra.
Also we prove that any GDN-Poisson algebra with the  identity $x\circ(y\cdot z)=(x\circ y )\cdot z
+(x\circ z) \cdot y$
is isomorphic to a commutative associative differential algebra.

\noindent{\bf Key words: }    Poincar\'{e}-Birkhoff-Witt theorem;
 \npa;  \cca.

\ \

\noindent {\bf AMS} Mathematics Subject Classification (2000): 16S10, 17A30, 17A50.



\section{Introduction}\label{Intro}
I.M. Gelfand and I.Ya. Dorfman \cite{Gelfand}, in connection with Hamiltonian operators in the formal calculus of variations, invented a new class of nonassociative algebras, defined by identities
\begin{eqnarray}\label{eq00}
\nonumber   x\circ(y\circ z)-(x\circ y)\circ z
&=&y\circ(x\circ z)-(y\circ x)\circ z\ \ (\mbox{left symmetry}),\\
      (x\circ y)\circ z&=&(x\circ z)\circ y  \  \  (\mbox{right commutativity})
\end{eqnarray}
S.I. Gelfand  (see \cite{Gelfand}) introduced an important subclass of new algebras (\ref{eq00}).
Namely, any associative commutative differential algebra $(C,\cdot)$
is an algebra (\ref{eq00}) under a new multiplication
$
       x\circ y= x(Dy).
$
Independently, S.P. Novikov (\cite{Novikov 1,Novikov 2}) invented the same algebras in connection with linear Poisson brackets of hydrodynamic type. J.M. Osborn \cite{ Osborn 1,   Osborn 2,   Osborn 3, Osborn 4}  gave
   the name Novikov algebra (he knew both papers \cite{Novikov 1,Gelfand}) to this kind of algebras and  began
   to classify simple  Novikov algebras
      as well as irreducible
    modules.

Considering the contributions of Gelfand and Dorfman to Novikov algebras, in this paper we call
Novikov algebras as Gelfand-Dorfman-Novikov algebras (GDN algebras for short),  Novikov-Poisson algebras as Gelfand-Dorfman-Novikov-Poisson  algebras (GDN-Poisson algebras for short).

  As it was pointed out in \cite{Novikov 2},     E.I. Zelmanov answered to a Novikov's question about
  simple finite dimensional GDN algebras over a field of characteristic $0$.
  He proved that there are no such algebras,   see  \cite{Zelmanov}.
  V.T. Filippov \cite{Filippov 1} found first examples of simple infinite dimensional GDN algebras
   of characteristic $p\geq 0$ and simple finite dimensional GDN algebras of characteristic $p> 0$.   J.M. Osborn and E.I. Zelmanov \cite{Osborn 5} classified simple GDN algebras $A$ over an algebraically closed field of
characteristic 0 with a maximal subalgebra $H$ such that $A /H$ has a finite dimensional irreducible $H$-submodule (V. Gullimen's type of condition)  (up to isomorphisms of B. Weisfeiler associated graded algebras).  A. Dzhumadil'daev and  C. L\"{o}fwall \cite{trees} constructed a linear basis of a free
GDN algebra using trees and a free commutative associative differential algebra. There were also  quite a few papers on the structure theory
(for example,   X. Xu \cite{Xiaoping Xu 1, Xiaoping Xu 2, Xiaoping Xu Poisson},
  C. Bai and D. Meng  \cite{Bai 1,Bai 3},      L. Chen,   Y. Niu and D. Meng \cite {Chen ly},
   D. Burde and K. Dekimpe \cite{Burde 3})
and combinatorial theory of GDN algebras,   and irreducible modules over GDN algebras,
 with applications to mathematics and mathematical physics.

In \cite{Xiaoping Xu 2, Xiaoping Xu Poisson, Xiaoping Xu conformal},
X. Xu introduced a concept of
GDN-Poisson algebras
in order to study the tensor theory of GDN algebras.
Then he classified GDN-Poisson algebras whose GDN algebras are simple with
an idempotent element. Moreover, he showed that a class of simple GDN algebras without
idempotent elements can be constructed through GDN-Poisson algebras. Certain
new simple Lie superalgebras induced by GDN-Poisson algebras are also introduced by him.
X. Xu  connected certain GDN-Poisson algebras with Hamiltonian superoperators
and proved that they induce Lie superalgebras, which are
natural generalizations of the super-Virasoro algebras in special cases.

Since
 then, there have been several papers on GDN-Poisson algebras, see \cite{Poisson Jordan, Bai poisson}.
There are also papers about embedding on GDN-Poisson algebras, for example \cite{Poisson embedding}; papers on GDN-Poisson algebras
and associative commutative derivation algebras \cite{Poisson Derivation}.

The paper is organized as follows.
In section 2 we first introduce the concept of a \cca\  and then construct a linear basis of a free \npa\ over an arbitrary  field.
In section 3 we  prove a Poincar\'{e}-Birkhoff-Witt (PBW for short) theorem: any \npa\ is embeddable into its universal enveloping \cca. In section 4 we show  that any \npa\ (with unit  $e$ with respective
 to $\cdot$) satisfying the identity $x\circ(y\cdot z)=(x\circ y )\cdot z
+(x\circ z) \cdot y$ is isomorphic to a commutative associative differential algebra both as \npa\ and as commutative associative differential algebra.

\section{Free \npas}\label{section free np}

A \textit{\na}  (\cite{Novikov 1}) $\mathcal{A}$ is a vector space with a bilinear  operation $\circ$ satisfying the two identities
$$
x\circ(y\circ z)-(x\circ y)\circ z
=y\circ(x\circ z)-(y\circ x)\circ z\ \ (\mbox{left symmetry}),
$$
$$
(x\circ y)\circ z=(x\circ z)\circ y  \  \  (\mbox{right commutativity}).
$$

A \textit{\npa}  (\cite{Xiaoping Xu 2}) $\mathcal{A}$ is
a vector space
with two bilinear operations $``\cdot"$ and $``\circ"$ such that $(\mathcal{A}, \cdot)$
forms a commutative associative  algebra
and $(\mathcal{A}, \circ)$ forms a \na\ with the compatibility conditions:
$$
(x\cdot y)\circ z= x\cdot (y\circ z),
$$
$$
(x\circ y)\cdot z-x\circ (y\cdot z)=(y\circ x)\cdot z-y\circ (x\cdot z), \  x, y, z\in \mathcal{A}.
$$

Throughout the paper, we only consider \npas\
with  unit $e$ with respect to $\cdot$.

As usual, let  $F(\Omega, X)$ be a free $\Omega$-algebra generated by $X$, where  $\Omega$ is a given set of operator symbols and $X$ is a non-vacuous set. The set  of words in $F(\Omega, X)$  is denoted by $W(\Omega, X)$ (for definitions of $\Omega$-algebra and words, see  \cite{jiapeng} and
Chapter 2 in \cite{jacobson 2}). For example, let $GDNP(X)$ be a free \npa\ over a field $k$ generated by $X$. Then
(i) $x, e$ are words in $GDNP(X)$ for any $x\in X$;
(ii) If $u, v$ are words in $GDNP(X)$, then
$\cdot u v$ and $\circ  uv$ are words in $GDNP(X)$.
For convenience, we will write $(u\delta v)$ instead of  $\delta u v$ if $\delta$ is a binary operator. Moreover, we always omit the leftmost ``(" and the rightmost ``)" of a word. For example, $((x_1\circ x_1)\cdot (x_1\circ e))\circ x_2$ is a word in $GDNP(X)$.

For any words $w_1,w_2,\dots, w_n\in  W(\Omega, X)$ and  binary operators $\delta_1,\ \delta_2,\ \dots, \delta_{n-1}\in \Omega$, define
$$
[w_{1}\delta_1 w_{2}\delta_2 \cdots\delta_{n-1}  w_{n}]_{_L}:=(\dots((w_{1}\delta_1 w_{2})\delta_2 w_{3})\delta_3 \cdots\delta_{n-2} w_{n-1})\delta_{n-1} w_{n} \
$$
(\mbox{left-normed bracketing}) and
$$
 [w_{1}\delta_1 w_{2}\delta_2 \cdots\delta_{n-1}  w_{n}]_{_R}:=w_{1}\delta_1( w_{2}\delta_2\cdots\delta_{n-3} (w_{n-2}\delta_{n-2} (w_{n-1}\delta_{n-1} w_{n}))\dots) \
$$(\mbox{right-normed bracketing}).

For any word $T$ in $GDNP(X)$, define
$$
|T|_{\circ}:= \mbox{the number of } ``\circ" \mbox{  that appears in }T,
$$
$$
|T|_{_X}:=\mbox{the number of elements in $X$ that appears in $T$}.
$$
For example, let $T=(x_1\circ x_1)\cdot (x_1\circ e)$. Then  $|T|_{_X}=3$, $|T|_{\circ}=2$.

 We call  $w$ a  \textit{GDN tableau} (\cite{trees})  over a well-ordered set $X$, if
$$
w=(\dots ((a_{1,r_1+1}\circ A_1 )\circ A_{2})\circ \cdots \circ  A_{n}) \mbox{ (left-normed bracketing),}
$$
    where
$
 A_{i}=(a_{i,  r_i}\circ \cdots \circ (a_{i,  3}\circ (a_{i,  2}\circ a_{i,  1}))\dots)\mbox{ (right-normed bracketing)},   \   1\leq i \leq n,\  a_{i,j}\in X, $
 satisfying that (i) $r_i\geq r_{i+1}$,
 (ii) $a_{i,  1}\geq a_{i+1,  1}$      if $r_{i}=r_{i+1},  $ $  i\geq 1 $,
(iii) $a_{1,  r_{1}+1}\geq a_{1,  r_{1}}\geq\dots \geq a_{1,  3}\geq a_{1,  2}\geq a_{2,  r_{2}}\geq\dots\geq a_{2,  3}\geq a_{2, 2} \geq a_{n,  r_{n}}\geq \dots  \geq a_{n,  3}\geq a_{n,  2} $.

Let $[X]$ be the commutative monoid generated by $X$ with unit $e$.
We call $T=u\cdot w$ a \textit{\npt} over $X$, if
$u=b_1\cdots b_m\in [X]$ (each $b_i\in X$)
 and $w=(\cdots ((a_{1,r_1+1}\circ A_1 )\circ A_{2})\circ \cdots \circ  A_{n})$
is a  \nt\ over $X\cup \{ e\}$ ($e<x$ for any $x\in X$) satisfying that
(i) $a_{n,2}\geq b_1\geq \dots \geq b_m$, (ii) if $a_{n,2}=e$, then $m=0$, i.e., $T=w$.

We call $[a_{1}\circ a_{2}\circ \dots\circ  a_{n}]_{_R}$  a \textit{row} in $GDNP(X)$ if $a_i\in [X]$, for all $1\leq i\leq n$.

We will prove that the set of  the \nptx\ over $X$ is a linear basis of  $GDNP(X)$.

Unless otherwise stated, throughout the paper, we   use $a,b,c, \dots $ to denote elements in $X\cup \{ e\}$, $\alpha, \beta,\gamma, \dots$
elements in the field $k$, $u,v, \dots$
words in the commutative associative monoid $[X]$, $T$  word in  $ GDNP(X)$, $m,n,p,q, t,l, i,j, \dots$  nonnegative integer numbers.
Moreover, we may omit $\cdot$,  if $\cdot$ is a commutative and associative bilinear operation. For example, $T_1\cdot T_2\cdot T_3$ (with any bracketing) in $GDNP(X)$ is simply denoted by $T_1T_2 T_3$.

\begin{definition}
\emph{A} \cca\ \emph{$(\mathcal{A},\cdot, \ast, D)$}  \emph{is
a vector space
with bilinear  operations} $\cdot, \ast $  \emph{and a linear operation $D$ such that} $(\mathcal{A}, \cdot)$
\emph{forms a commutative associative  algebra with unit} $e$, $(\mathcal{A}, \ast)$ \emph{forms a commutative associative  algebra and}
$``\cdot, \ast , D"$ \emph{are compatible in the sense that the following identities hold}:
\begin{eqnarray*}
&&(x\cdot y)\ast z=x \cdot (y \ast z),\\
&&D(x\ast y)=(Dx)\ast y +x \ast (Dy),\\
&&D(x\cdot y) =(Dx) \cdot y +x \cdot (Dy)-x\cdot y \cdot (De).
\end{eqnarray*}
\end{definition}

Let $(\mathcal{A}, \cdot, \ast, D)$ be a  \cca. Then for any $x,y,z\in \mathcal{A}$, we have
 $$
 (x\cdot y)\ast z=x \cdot (y \ast z),
 $$
 $$
(x\ast y)\cdot z=(y\ast x)\cdot z=z\cdot (y\ast x)=(z\cdot y)\ast x=x\ast (z\cdot y)=x\ast (y\cdot z).
$$
Therefore, for any $y_1, \dots, y_n \in \mathcal{A}$, $\delta_1, \dots, \delta_{n-1}\in \{\cdot, \ast\}$, we have
$[y_1\delta_1y_2\delta_2\cdots\delta_{n-1}y_n]_{_L}=(y_1\delta_1y_2\delta_2\cdots\delta_{n-1}y_n)$ (with any bracketing), and thus $y_1\delta_1y_2\delta_2\cdots\delta_{n-1}y_n$ makes sense.

\begin{lemma}\label{become novikov poisson}
Let $(\mathcal{A}, \cdot, \ast, D)$ be a  \cca.
The operation $\circ$  on \ma is defined by
$$
x\circ y := x\ast Dy, \ x,y\in \mathcal{A}.
$$
Then $(\mathcal{A}, \cdot, \circ)$ forms a \npa.
\end{lemma}

\begin{proof}
The proof is straightforward.  For example,
\begin{eqnarray*}
&&(x\circ y ) \cdot z-x\circ(y\cdot z)\\
&=&(x\ast Dy)\cdot z -x\ast (Dy\cdot z)-x\ast (y \cdot Dz)+x\ast (y\cdot z\cdot De)\\
&=&-x\ast y \cdot Dz +x\ast y\cdot z\cdot De\\
&=&(y\circ x) \cdot z-y\circ(x\cdot z).
\end{eqnarray*}
The other identities can be proved similarly.
\end{proof}

Define  $D^0a:=a$, $a\in X\cup \{e\}$.
Let $ Y=\{ D^{i}a, \cdot, \ast \mid i \in \mathbb{Z}_{\geq 0},a\in X\cup \{e\}\}$ and  $Y^{+}$ be the free semigroup generated by $Y$,
\begin{eqnarray*}
&C[X]:=&\{D^{i_{1}}a_1\ast \dots  \ast D^{i_{j}}a_j\cdot D^{i_{j+1}}a_{j+1} \cdot D^{i_{j+2}}a_{j+2}\cdots D^{i_{n}}a_{n} \in Y^{+}\mid n\geq 1, 1\leq j\leq n, \\
&&(i_{1},a_1) \geq \dots\geq  (i_{n},a_n) \mbox{ lexicographically, and if } j<n, \mbox{ then } (i_{n},a_n) \neq (0,e)\},
\end{eqnarray*}
where $\mathbb{Z}_{\geq 0}$ is the set of nonnegative integer numbers.

For any $w_1=D^{i_{1}}a_1\ast \dots  \ast D^{i_{n}}a_n \cdot D^{i_{n+1}}a_{n+1} \cdots D^{i_{n+m}}a_{n+m},\  w_2=D^{j_{1}}b_1\ast \dots  \ast D^{j_{p}}b_p \cdot D^{j_{p+1}}b_{p+1} \cdots D^{j_{p+q}}b_{p+q} \in C[X]$, define
$w_1=w_2$ if and only if $n=p, m=q, (i_l, a_l)=(j_l,b_l)$ for all $1\leq l\leq n+m.$

Let $kC[X]$ be the linear space with a $k$-basis $C[X]$ and $w_1, w_2$ be as above.
Let $ D^{l_{1}}c_1,  \dots, D^{l_{n+m+p+q}}c_{n+m+p+q}$ be a reordering of $D^{i_{1}}a_1$, $\dots$,   $D^{i_{n+m}}a_{n+m}$, $D^{j_{1}}b_1$, $\dots$,  $D^{j_{p+q}}b_{p+q}$, such that  $(l_{1},c_1)\geq \dots \geq (l_{n+m+p+q},c_{n+m+p+q})$ lexicographically. For convenience, we define $(l_{n+m+p+q+1},c_{n+m+p+q+1})=(0,e)$.

We define the  bilinear operations $\cdot, \ast$ and the linear operation $D$ on $kC[X]$  as follows:
\begin{enumerate}
 \item[(i)]\  $  w_1\cdot w_2:=D^{l_{1}}c_1\ast \dots  \ast D^{l_{n+p-1}}c_{n+p-1} \cdot  D^{l_{n+p}}c_{n+p} \cdots D^{l_{t}}c_{t}$,
where $t=n+p-1$ if $(l_{n+p},c_{n+p})=(0,e)$; otherwise,   $n+p\leq t\leq n+m+p+q, \ (l_{t}, c_{t})>(0,e)=(l_{t+1}, c_{t+1})$.
 \item[(ii)]\ $  w_1\ast w_2:=D^{l_{1}}c_1\ast \dots  \ast D^{l_{n+p}}c_{n+p} \cdot  D^{l_{n+p+1}}c_{n+p+1} \cdots D^{l_{t}}c_{t}
$,
where  $t=n+p$ if $ (l_{n+p+1},c_{n+p+1})=(0,e)$; otherwise, $n+p+1\leq t\leq n+m+p+q, \ (l_{t}, c_{t})>(0,e)=(l_{t+1}, c_{t+1})$.
\item[(iii)]\ $ D(w_1):=\sum_{1\leq t\leq n+m} [D^{i_{1}}a_1\delta_1 D^{i_{2}}a_2\delta_2\cdots  \delta_{t-1}  D^{i_{t}+1}a_t\delta_{t} \cdots  \delta_{n+m-1} D^{i_{n+m}}a_{n+m}]_{_L}-m w_1\cdot De$,
where   $ \delta_{1}=\dots = \delta_{n-1}=\ast $ and $\delta_{n}=\dots = \delta_{n+m-1}=\cdot $.
\end{enumerate}

With notations as above, we have
\begin{lemma}
$(kC[X], \cdot, \ast, D)$ is a \cca.
\end{lemma}
\begin{proof}
By definitions of $\cdot$ and $\ast$, we have $[D^{i_{1}}a_1\ast \dots  \ast D^{i_{n}}a_n \cdot D^{i_{n+1}}a_{n+1} \cdots D^{i_{n+m}}a_{n+m}\cdot e]_{_L}=w_1=w_1\cdot e=e\cdot w_1$ and $  w_1\cdot w_2=[D^{l_{1}}c_1\ast \dots  \ast D^{l_{n+p-1}}c_{n+p-1} \cdot  D^{l_{n+p}}c_{n+p} \cdots D^{l_{n+m+p+q}}c_{n+m+p+q}]_{_L}$. It follows immediately that $\cdot$ is associative and commutative. Similarly,
$\ast$ is associative and commutative. Moreover, the identity $x \cdot (y \ast z)=(x\cdot y)\ast z$ follows by the above arguments. Let $D^{i_{n+m+1}}a_{n+m+1}=D^{i_{n+m+2}}a_{n+m+2}=\cdots=D^{i_{n+m+l}}a_{n+m+l}=e$. Then
\begin{eqnarray*}
&&D([w_1\cdot D^{i_{n+m+1}}a_{n+m+1}\cdot D^{i_{n+m+2}}a_{n+m+2} \cdots  D^{i_{n+m+l}}a_{n+m+l}]_{_L})=D(w_1)\\
&=&[D(w_1)\cdot  D^{i_{n+m+1}}a_{n+m+1}\cdot D^{i_{n+m+2}}a_{n+m+2} \cdots  D^{i_{n+m+l}}a_{n+m+l}]_{_L}\\
&&+\sum_{n+m+1\leq t\leq n+m+l}[w_1\cdot  D^{i_{n+m+1}}a_{n+m+1}\cdots D^{i_{t}+1}a_{t}\cdots D^{i_{n+m+l}}a_{n+m+l}]_{_L}-lw_1\cdot De\\
&=&\sum_{1\leq t\leq n+m+l} [D^{i_{1}}a_1\delta_1\cdots  \delta_{t-1}  D^{i_{t}+1}a_t\delta_{t} \cdots  \delta_{n+m+l-1} D^{i_{n+m+l}}a_{n+m+l}]_{_L}\\
&&-(m+l)[w_1\cdot D^{i_{n+m+1}}a_{n+m+1}\cdot D^{i_{n+m+2}}a_{n+m+2} \cdots  D^{i_{n+m+l}}a_{n+m+l}]_{_L}\cdot De,
\end{eqnarray*}
 where   $ \delta_{1}=\dots = \delta_{n-1}=\ast $ and $\delta_{n}=\dots = \delta_{n+m+l-1}=\cdot $.
 Therefore,
\begin{eqnarray*}
&&D(w_1\cdot w_2)\\
&=&D([D^{l_{1}}c_1\ast \dots  \ast D^{l_{n+p-1}}c_{n+p-1} \cdot  D^{l_{n+p}}c_{n+p} \cdots D^{l_{n+m+p+q}}c_{n+m+p+q}]_{_L})\\
&=&\sum_{1\leq t\leq n+m+p+q} [D^{l_{1}}c_1\delta_1 D^{l_{2}}c_2\delta_2\cdots  \delta_{t-1}  D^{l_{t}+1}c_t\delta_{t} \cdots  \delta_{n+m+p+q-1} D^{l_{n+m+p+q}}c_{n+m+p+q}]_{_L}\\
&&-(m+q+1) w_1\cdot w_2\cdot De\\
&=&\sum_{1\leq t\leq n+m} [D^{i_{1}}a_1\eta_1 D^{i_{2}}a_2\eta_2\cdots  \eta_{t-1}  D^{i_{t}+1}a_t\eta_{t} \cdots \eta_{n+m-1} D^{i_{n+m}}a_{n+m}]_{_L}\cdot w_2-m w_1\cdot De\cdot w_2\\
&&+\sum_{1\leq t\leq p+q} [D^{j_{1}}b_1\nu_1 D^{j_{2}}b_2\nu_2\cdots  \nu_{t-1}  D^{j_{t}+1}b_t\nu_{t} \cdots  \nu_{p+q-1} D^{j_{p+q}}b_{p+q}]_{_L}\cdot w_1-q w_2\cdot De \cdot w_1\\
&&-w_1\cdot w_2\cdot De\\
&=&D(w_1)\cdot w_2+D(w_2)\cdot w_1-w_1\cdot w_2\cdot De,
\end{eqnarray*}
where $\delta_{1}=\dots = \delta_{n+p-2}=\eta_1=\dots=\eta_{n-1}=\nu_1=\dots=\nu_{p-1} =\ast$ , $\delta_{n+p-1}=\dots = \delta_{n+m+p+q-1}=\eta_{n}=\dots=\eta_{n+m-1}=\nu_{p}=\dots=\nu_{p+q-1}=\cdot$.
\begin{eqnarray*}
&&D(w_1\ast w_2)\\
&=&D([D^{l_{1}}c_1\ast \dots  \ast D^{l_{n+p}}c_{n+p} \cdot  D^{l_{n+p+1}}c_{n+p+1} \cdots  D^{l_{n+m+p+q}}c_{n+m+p+q}]_{_L})\\
&=&\sum_{1\leq t\leq n+m+p+q} [D^{l_{1}}c_1\delta_1 D^{l_{2}}c_2\delta_2\cdots  \delta_{t-1}  D^{l_{t}+1}c_t\delta_{t} \cdots  \delta_{n+m+p+q-1} D^{l_{n+m+p+q}}c_{n+m+p+q}]_{_L}\\
&&-(m+q) w_1\cdot w_2\cdot De,\\
&=&\sum_{1\leq t\leq n+m} [D^{i_{1}}a_1\eta_1 D^{i_{2}}a_2\eta_2\cdots  \eta_{t-1}  D^{i_{t}+1}a_t\eta_{t} \cdots \eta_{n+m-1} D^{i_{n+m}}a_{n+m}]_{_L}\ast w_2-m w_1\cdot De\ast w_2\\
&&+\sum_{1\leq t\leq p+q} [D^{j_{1}}b_1\nu_1 D^{j_{2}}b_2\nu_2\cdots  \nu_{t-1}  D^{j_{t}+1}b_t\nu_{t} \cdots  \nu_{p+q-1} D^{j_{p+q}}b_{p+q}]_{_L}\ast w_1-q w_2\cdot De \ast w_1\\
&=&D(w_1)\ast w_2+D(w_2)\ast w_1,
\end{eqnarray*}
where $\delta_{1}=\dots = \delta_{n+p-1}=\eta_1=\dots=\eta_{n-1}=\nu_1=\dots=\nu_{p-1} =\ast$ , $\delta_{n+p}=\dots = \delta_{n+m+p+q-1}=\eta_{n}=\dots=\eta_{n+m-1}=\nu_{p}=\dots=\nu_{p+q-1}=\cdot$.

It follows that $(kC[X], \cdot, \ast, D)$ is a \cca.
\end{proof}

Let $C(X)$ be a free \cca\ generated by $X$ and $C[X]^{\prime}:=\{[w]_{_L}\mid w\in C[X]\}$. Then we have

\begin{lemma}
The set $C[X]^{\prime}$ is a linear generating set of $C(X)$.
\end{lemma}
\begin{proof}
Given a word  $w\in C(X)$, we can apply $D(x\ast y)=Dx\ast y +x \ast Dy$ and
$D(x\cdot y) =Dx \cdot y +x \cdot Dy-x\cdot y \cdot De$ (if any) to rewrite $w$ as a linear combination of words of the form $(D^{i_1}a_1\delta_1D^{i_2}a_2\delta_2\cdots\delta_{n-1}D^{i_n}a_n)$ (with some bracketing), where $n\geq 1, a_1,\dots, a_n\in X\cup\{e\},i_1,\dots i_n\in \mathbb{Z}_{\geq 0}$.
Noting that for any $x, y, z\in C(X)$,
$
(x\cdot y)\ast z=x\cdot(y\ast z)
$ and
$
(x\ast y)\cdot z=x\ast (y\cdot z)
$,
 we have
$[D^{i_1}a_1\delta_1D^{i_2}a_2\delta_2\cdots\delta_{n-1}D^{i_n}a_n]_{_L}=(D^{i_1}a_1\delta_1D^{i_2}a_2\delta_2\cdots\delta_{n-1}D^{i_n}a_n)$ (with any bracketing). Since $[x\cdot y \ast z]_{_L}=(x\cdot y)\ast z=z\ast (x\cdot y)=(z\ast x)\cdot y=(x\ast z)\cdot y=x\ast (z\cdot y)=x\ast (y\cdot z)=(x\ast y)\cdot z=[x\ast y \cdot z]_{_L}$, we may assume that, in $[D^{i_1}a_1\delta_1D^{i_2}a_2\delta_2\cdots\delta_{n-1}D^{i_n}a_n]_{_L}$, $\delta_1=\delta_2=\dots =\delta_l=\ast$,
$\delta_{l+1}=\delta_{l+2}=\dots =\delta_{n-1}=\cdot$ for some $0\leq l\leq n-1$. Finally, by applying  the commutativity of $\cdot$ and $ \ast$, we get that any element of $C(X)$ can be written as a linear combination of elements of $C[X]^{\prime}$.
\end{proof}

\begin{lemma} The algebra
$(kC[X], \cdot, \ast, D)$ is isomorphic to $C(X)$.
\end{lemma}

\begin{proof}
Define a \cca\ homomorphism $\eta \ : \ C(X)\longrightarrow kC[X]$ induced by $\eta (a)=a $ for any $a\in X$.  Since $C[X]^{\prime}$ is a linear generating set of $C(X)$
 and $\eta([w]_{_L})=w$ for any
 $w\in C[X]$ and $C[X]$ is a linear basis of $kC[X]$, we have that $C[X]^{\prime}$ is a linear basis of $C(X)$.  It follows that   $\eta$ is an isomorphism.
\end{proof}
From now on, we identify $C[X]^{\prime}$ with $C[X]$.

Let $w=D^{i_{1}}a_1\ast \dots  \ast D^{i_{j}}a_j\cdot D^{i_{j+1}}a_{j+1} \cdot D^{i_{j+2}}a_{j+2}\cdots D^{i_{n}}a_{n}\in C[X]$. Define
$$
 wt(w):=(\sum_{1\leq t\leq n}i_t)-(j-1),\ \
|w|_{\ast}:=j-1,\ \ ord(w):=(|w|_{\ast}, i_1, a_{1},\dots , i_n,  a_{n},-1).
$$
We order the set of $C[X]$ as follows:
$$
w_1<w_2 \Leftrightarrow ord(w_1)<ord(w_2)\mbox{ lexicographically}.
$$

The proof of the following lemma is straightforward.

\begin{lemma}\label{keep D}
\begin{enumerate}
\item[(i)]\ For any $w_1, w_2, w_3\in  C[X]$, we have
$$
w_1<w_2  \Rightarrow w_1\ast w_3 < w_2\ast w_3 ,\  w_1\cdot w_3 < w_2\cdot w_3.
$$
\item[(ii)]\ Let $w=D^{i_{1}}a_1\ast \dots  \ast D^{i_{j}}a_j\cdot D^{i_{j+1}}a_{j+1} \cdot D^{i_{j+2}}a_{j+2}\cdots D^{i_{n}}a_{n}$. If $(i_1,a_1)>(i_2, a_2)
    $, then
    $$
    \overline{D(w)}=D^{i_{1}+1}a_1\ast \dots  \ast D^{i_{j}}a_j\cdot D^{i_{j+1}}a_{j+1} \cdot D^{i_{j+2}}a_{j+2}\cdots D^{i_{n}}a_{n},
    $$
     where $\overline{f}$ is the leading word of $f \in kC[X]$.

\end{enumerate}
\end{lemma}

Define
$$
  GDNP_{_0}(X):= span_{k}\{w\in C[X] \mid wt(w)=0\},
$$
which is the subspace of $C(X) $ spanned by the set $\{w\in C[X] \mid wt(w)=0\}$.
Then it is clear that $ (GDNP_{_0}(X),\cdot, \circ)$ is a GDN-Poisson subalgebra of  $(C(X), \cdot, \circ)$.
We will prove that $GDNP_{_0}(X)$ is a free \npa\ generated by $X$, i.e.,  $GDNP(X)\cong GDNP_{_0}(X)$, see Lemma \ref{isomorphism}.

\begin{definition}\label{def2.2}
\emph{Let $\Omega=\{\cdot, \circ\}$, where $\cdot, \circ$ are binary operator symbols. A} root function \emph{from the set of words of a free  $\Omega$-algebra $F(\Omega, X\cup \{e\})$
to   $\mathbb{Z}_{\geq 0}$ is defined as follows: for any words $T_1, T_2$,}
 \begin{enumerate}
 \item[\emph{(i)}]\  \emph{$r(a)=0$ for any $a\in X\cup \{e\}$};
\item[\emph{(ii)}]\  \emph{$r(T_1\cdot T_2)=r(T_1)+r(T_2)$};
\item[\emph{(iii)}]\  \emph{$r(T_1\circ T_2)=r(T_1)+1$ if $|T_2|_{\circ}=0$, and
         $r(T_1\circ T_2)=r(T_1)+r(T_2)$  if $|T_2|_{\circ}\geq 1$}.
 \end{enumerate}
\end{definition}
We call $r(T)$ the \textit{root number} of $T$. It is clear that $r(T_1\circ T_2)=r(T_1)+max\{1, r(T_2)\}$.

For any word $T\in  GDNP(X)$, we can
take $T$ as a word in $F(\Omega, X\cup \{e\})$. Then $r(T)$ also makes sense.

\begin{example}{\label{example n-1}}
For any $u_i\in [X], i\geq 1, n\geq 1$, we have
 $r([u_1\circ u_2\circ \cdots\circ u_n]_{_L})=n-1$.
\end{example}

By definition, we have the following lemmas.

\begin{lemma}{\label{right commutative keep} }
For any words  $x,y,z\in F(\Omega, X\cup\{e\})$,  we have
$r((x\circ y)\circ z)=r((x\circ z)\circ y)$ and
$r((x\cdot y)\circ z)=r(x\cdot (y\circ z))$.
\end{lemma}

\begin{lemma}{\label{multiply keep}}
For any words $T_1,T_2,T\in F(\Omega, X\cup\{e\})$, if
$r(T_1)>r(T_2)$,
then we have
\begin{enumerate}
 \item[(i)]\ $r(T_1\circ T)>r(T_2\circ T)$;
 \item[(ii)]\  $r(T\circ T_1)>r(T\circ T_2)$  if $ r(T_1)>1$, $r(T\circ T_1)=r(T\circ T_2)$  if $ r(T_1)=1$.
 \end{enumerate}
\end{lemma}

\begin{lemma}\label{as circle}
For any word $T\in GDNP(X)$, if $|T|_{\circ}\geq 1$, then $T=T_1\circ T_2$ for some
words $T_1, T_2\in GDNP(X)$.
\end{lemma}

\begin{proof}
If $T=T_1\circ T_2$, the result follows. Otherwise, we may assume that $T=T_1 T_2\cdots T_l$, where $|T_1|_{\circ}\geq\cdots \geq |T_l|_{\circ}$  and $T_1=T_{11}\circ T_{12}$. Then $T=(T_{11} T_2\cdots T_l)\circ T_{12}$.
\end{proof}

\begin{lemma}{\label{prop n-1}}
For any word $T\in GDNP(X)$ with $|T|_{\circ}= n-1$, we have
$
r(T)\leq |T|_{\circ}
$
and
the equality holds
if and only if
$T=[u_1\circ u_{2}\cdots \circ u_{n}]_{_L}
$  for some
$u_1,u_2,\dots , u_{n}\in [X]$.
\end{lemma}

\begin{proof}
By Example \ref{example n-1}, we just need to prove that (i) $r(T)\leq |T|_{\circ}$,
(ii) if $r(T)=|T|_{\circ}$, then
$T=[u_1\circ u_{2}\cdots \circ u_n]_{_L}
$, where each
$u_i\in [X]$.
Induction on $|T|_{\circ}$. If $|T|_{\circ}\leq 1$, the result follows by Lemma \ref{as circle}.
Let $|T|_{\circ}=n>1$ and $T=T_1\circ T_2.$
If $|T_2|_{\circ}\geq 1$, then by induction we have
$
r(T)=r(T_1)+r( T_2)\leq |T_1|_{\circ} +|T_2|_{\circ}=|T|_{\circ}-1<|T|_{\circ}.
$
If $|T_2|_{\circ}=0 $, then by induction we have
$
r(T)=r(T_1)+1\leq |T_1|_{\circ} +1=|T|_{\circ}
$.
Moreover,
$r(T)=|T|_{\circ}$
implies that $T_2=u_n$ and $ T_1=[u_1\circ u_{2}\circ \cdots \circ u_{n-1}]_{_L}$ by induction.
\end{proof}

For any words $T,T'\in GDNP(X)$,
we define
$$
T\rightarrow  T' \ \ \  \mbox{if} \ \ \  T=T'+\sum_{1\leq i\leq  n}\alpha_{i}T_{i}
$$
for some $n\in \mathbb{Z}_{\geq 0}$, words $T_i\in GDNP(X),\alpha_{i}\in k $, such that for any $ 1\leq i\leq  n$, $|T|_{\circ}=|T'|_{\circ}=|T_i|_{\circ},\ |T|_{_X}=|T'|_{_X}=|T_i|_{_X},\  r(T)=r(T')<r(T_i)$.

It is clear that $T\rightarrow  T'\rightarrow  T''\Rightarrow T\rightarrow  T''$.

 For any
 $A=[u_{r}\circ\cdots\circ u_2\circ u_1]_{_R}\in GDNP(X), u_1,\dots, u_r, v_j\in [X]$,
we  define
\begin{eqnarray*}
 A_{\hat{u}_i}&:=& [u_{r}\circ\cdots\circ  u_{i+1}\circ u_{i-1}\circ\cdots\circ u_2\circ u_1]_{_R},\\
 A_{u_i\mapsto v_j} &:=&[u_{r}\circ\cdots\circ  u_{i+1}\circ v_{j}\circ u_{i-1}\circ\cdots\circ u_2\circ u_1]_{_R},\\
 A_{u_i\leftrightarrow u_j} &:=&[u_{r}\circ\cdots\circ  u_{i+1}\circ u_{j}\circ u_{i-1}\circ\cdots\circ u_{j+1}\circ u_{i}\circ u_{j-1} \circ\cdots\circ u_1]_{_R}.
\end{eqnarray*}

\begin{lemma}\label{interchange leaves}
For any
$
A=[u_r\circ \cdots\circ u_2\circ u_1]_{_R}\in GDNP(X), r\geq 3, u_1, \dots, u_r\in [X],
$
we have
$A\rightarrow  A_{u_i\leftrightarrow u_j} $
   and $A\rightarrow  u_j\circ A_{\hat{u}_j}$ for any $r\geq i>j\geq 2$.
\end{lemma}

\begin{proof}
Since
\begin{eqnarray*}
A&=&[u_r\circ\cdots\circ (u_{j+1}\circ u_{j})\circ u_{j-1}\circ\cdots\circ u_1]_{_R}
+[u_r\circ\cdots\circ u_{j}\circ u_{j+1}\circ u_{j-1}\circ\cdots\circ  u_1]_{_R}\\
&&-[u_r\circ\cdots\circ(u_{j}\circ u_{j+1})\circ u_{j-1}\circ\cdots\circ  u_1]_{_R},
\end{eqnarray*}
we have
$
A\rightarrow [u_r\circ\cdots\circ u_{j}\circ u_{j+1}\circ u_{j-1}\circ\cdots\circ  u_1]_{_R}$. By a similar argument, we have
\begin{eqnarray*}
A&\rightarrow&[u_r\circ\cdots\circ u_{j}\circ u_{j+1}\circ u_{j-1}\circ\cdots\circ  u_1]_{_R}\\
&\rightarrow&\cdots\\
&\rightarrow&[u_r\circ\cdots\circ u_{i+1}\circ u_{j}\circ u_{i}\circ u_{i-1}\circ\cdots u_{j+1}\circ u_{j-1}\circ\cdots\circ  u_1]_{_R}\\
&\rightarrow&[u_r\circ\cdots\circ u_{i+1}\circ u_{j}\circ u_{i-1}\circ u_{i}\circ\cdots u_{j+1}\circ u_{j-1}\circ\cdots\circ  u_1]_{_R}\\
&\rightarrow&\cdots\\
&\rightarrow&[u_r\circ\cdots\circ u_{i+1}\circ u_{j}\circ u_{i-1}\circ\cdots u_{j+1}\circ u_{i}\circ u_{j-1}\circ\cdots\circ  u_1]_{_R}.
\end{eqnarray*}
Similarly, we have $A\rightarrow  u_j\circ A_{\hat{u}_j}$ for any $j\geq 2$.
\end{proof}

\begin{lemma}\label{interchange two row leaves}
Let
$A=[u_{r+1}\circ\cdots\circ u_2\circ u_1]_{_R}$,
$B= [v_m\circ\cdots\circ v_2\circ v_1]_{_R} \in GDNP(X)
$.
Then
$
A\circ B\rightarrow A_{u_i\mapsto v_j}\circ B_{v_j\mapsto u_i}.
$
Moreover, let
$
T=b_1\cdots b_m\cdot[u_{1,r_1+1}\circ A_1 \circ A_{2}\circ \cdots \circ  A_{n}]_{_L},
$
where each
$
A_{i}=[u_{i,  r_i}\circ \cdots \circ   u_{i,  2}\circ u_{i,  1}]_{_R},
   \  1\leq i \leq n
$.
Then $T\rightarrow (T)_{u_{i,j}\leftrightarrow u_{p,l}},\  j,l\geq 2$ and $T\rightarrow (T)_{b_{t}\leftrightarrow u_{i,j}},\  j\geq 2,\ 1\leq t\leq m$.
\end{lemma}

\begin{proof}
By Lemma \ref{interchange leaves}, we get
$$
A\circ B\rightarrow (u_i\circ A_{\hat{u}_i})\circ B
\rightarrow(u_i\circ B)\circ A_{\hat{u}_i}
\rightarrow
(v_j\circ B_{v_j\mapsto u_i})\circ A_{\hat{u}_i}
\rightarrow
(v_j\circ A_{\hat{u}_i})\circ B_{v_j\mapsto u_i}
\rightarrow
A_{u_i\mapsto v_j}\circ B_{v_j\mapsto u_i}.
$$
Since
$
T=b_1\cdots b_m\cdot [u_{1,r_1+1}\circ A_1 \circ A_{2}\circ \cdots \circ  A_{n}]_{_L}
=b_1\cdots b_m\cdot [u_{1,r_1+1}\circ A_i \circ A_{p}\circ A_1 \circ \cdots \circ \hat{A_{i}}\circ \cdots \circ \hat{A_{p}}\circ \cdots \circ  A_{n}]_{_L}
$ and
\begin{eqnarray*}
T&=&b_1\cdots\hat{b_t}\cdots b_m u_{1,r_{1}+1}\cdot [b_t \circ A_1 \circ A_{2}\circ \cdots \circ  A_{n}]_{_L}\ (\mbox{if } (i,j)=(1,r_1+1), \mbox{ we are done.} )\\
&\rightarrow& b_1\cdots\hat{b_t}\cdots b_m u_{1,r_{1}+1}\cdot [b_t \circ A_i  \circ A_1 \circ \cdots \circ \hat{A_{i}}\circ \cdots \circ  A_{n}]_{_L}\\
&\rightarrow& b_1\cdots\hat{b_t}\cdots b_m u_{1,r_{1}+1}\cdot [u_{i,j} \circ (A_i)_{u_{i,j}\mapsto b_t}  \circ A_1 \circ \cdots \circ \hat{A_{i}}\circ \cdots \circ  A_{n}]_{_L}\\
&\rightarrow& b_1\cdots\hat{b_t}\cdots b_m u_{i,,j}\cdot [u_{1,r_1+1} \circ (A_i)_{u_{i,j}\mapsto b_t}  \circ A_1 \circ \cdots \circ \hat{A_{i}}\circ \cdots \circ  A_{n}]_{_L}\\
&\rightarrow& b_1\cdots\hat{b_t}\cdots b_m u_{i,,j}\cdot [u_{1,r_1+1} \circ A_1 \circ \cdots \circ (A_i)_{u_{i,j}\mapsto b_t} \circ \cdots \circ  A_{n}]_{_L},
\end{eqnarray*}
the result follows immediately.
\end{proof}

Define  the \textit{deg-lex order} on $[X]$ as follows:
for any $u=x_1\cdots x_n,\ v=y_1\cdots y_m\in [X],$ $x_1\geq \cdots \geq x_n,\ y_1\geq \cdots \geq y_m,$  where each $x_i, \ y_j\in X$,
$
u>v  \mbox{ if and only if }n>m,  \mbox{ or }
 n=m \mbox{ and for some } p\leq n, x_p>y_p, x_i=y_i \mbox{ for all } 1\leq i\leq p-1.$

\begin{lemma}\label{Novikov tableaux}
For any word $T\in GDNP(X)$, we have $T=\sum \alpha_{l}T_{l}$,
where each $T_l$ is of the following form:
$
T_l=[u_{1,r_1+1}\circ A_1 \circ A_{2}\circ \cdots \circ  A_{n}]_{_L}
$, where each
$
A_{i}=[u_{i,  r_i}\circ \cdots \circ   u_{i,  2}\circ u_{i,  1}]_{_R},
  \   1\leq i \leq n,
 $   $u_{i,j}\in [X]$,
such that the following conditions hold:
\begin{enumerate}
 \item [(i)]\ $r(T_l)\geq r(T), \ |T_l|_{\circ}=|T|_{\circ},\ |T_l|_{_X}=|T|_{_X}$;

 \item [(ii)]\ $r_1\geq  \cdots \geq r_n$;

 \item [(iii)]\  If $r_i=r_{i+1}$, then $ u_{i,1}\geq u_{i+1,1}$ by deg-lex order on $[X]$;
 \item [(iv)]\  $u_{1,r_{1}+1}\geq \cdots \geq u_{1,2}\geq u_{2,r_{2}}\geq \cdots \geq u_{2,2}\geq \cdots \geq u_{n,r_{n}}\geq \cdots \geq u_{n,2}$ by deg-lex order on $[X]$.
 \end{enumerate}
\end{lemma}

\begin{proof}
Induction on $|T|_{\circ}$. If $|T|_{\circ}=0 \mbox{ or } 1$,
then the result follows by Lemma \ref{as circle}.
Suppose that the result holds for any $T$ with $|T|_{\circ}<t$.
 Let $|T|_{\circ}=t\geq 2$.
 Induction on $|T|_{\circ}-r(T)$.

  If $|T|_{\circ}-r(T)=0$, then $r(T)=t\geq 2$.
 By Lemma \ref{prop n-1},
 we get
$T=[u_1\circ u_{2}\circ \cdots \circ u_{t+1}]_{_L}$.
By right commutativity, the result follows.

If $|T|_{\circ}-r(T)\geq 1$, then by Lemma \ref{as circle}, we have
$T=T_1\circ T_2$ for some words $T_1,T_2\in GDNP(X)$.
By induction, we have
$T_1=\sum \alpha_{i}T_{i},\ T_2=\sum \beta_{j}T_{j}$, where $\alpha_{i}, \beta_{j}\in k$ and
 $T_{i}, T_{j}$ satisfy the claimed conditions.
Then
$
r(T_i\circ T_j)\geq r(T_i\circ T_2)\geq r(T_1\circ T_2)=r(T)
$
by Lemma \ref{multiply keep}.
Say
$
T_i=[u\circ A_1  \circ A_{2} \circ \cdots \circ  A_{p}]_{_L},
$
$
T_j=[v\circ B_1  \circ B_{2} \circ \cdots \circ  B_{q}]_{_L},
$
where $p,q\geq 0,$ and each $ A_m, B_{m^\prime}$ is a  row.

If $|T_{i}|_{\circ}\geq 1,$
then $p\geq 1$. Induction on $q$.
If $q=0$ or $1$, then by right commutativity,
Lemmas \ref{right commutative keep},\ \ref{interchange leaves} and  \ref{interchange two row leaves} and induction hypothesis, the result follows.
If $q>1$, then
$
T_i\circ T_{j}=[u\circ A_1  \circ A_{2} \circ \cdots \circ  A_{p-1} \circ T_j \circ A_p]_{_L}.
$
By induction,
$
[u\circ A_1  \circ A_{2} \circ \cdots \circ  A_{p-1} \circ T_j]_{_L}=
\sum\alpha_{l}^{\prime}T_l^{\prime}, \ r(T_l^{\prime})\geq p-1+q
$.
So
$
r(T_l^{\prime}\circ A_p)\geq p+q=r(T_i\circ T_j)\geq r(T).
$
Since
$
|T_l^{\prime}|_{\circ}\geq 1$ and $ r(A_p)=1,
$
the result follows by the above argument.

If $|T_i|_{\circ}=0$, then
$
|T_j|_{\circ}\geq 1,\  q\geq 1.
$
If $q=1$, by Lemma \ref{interchange leaves} and induction hypothesis, the result follows.
If $q>1$, then
$
T_i\circ T_{j}
=(T_{i}\circ [v\circ B_1  \circ B_{2} \circ \cdots \circ B_{q-1}]_{_L})\circ B_{q}
+
[v\circ B_1  \circ B_{2} \circ \cdots \circ B_{q-1}]_{_L}\circ(T_i\circ B_q)
-
[v\circ B_1  \circ B_{2} \circ \cdots \circ  B_{q-1}\circ T_i]_{_L}\circ B_q$. By induction and the above argument for the case $|T_i|_{\circ}\geq1$, the result follows.
\end{proof}

\begin{lemma}\label{root 1}
For any $u\in [X]$, $a_1,a_2, \dots ,a_n\in X, n\geq 1$,  we have
$$
u\circ(a_1\cdots a_n)=\sum_{1\leq i\leq n}(u a_1\cdots \hat{a_i}\cdots a_n)\circ a_i  -(n-1)(u a_1\cdots a_n)\circ e.
$$
\end{lemma}

\begin{proof}
Induction on $n$. If $n=1$, the result follows. If $n>1$,  then
\begin{eqnarray*}
&&u\circ(a_1\cdots a_n)=u\cdot (e \circ(a_1\cdots a_n) )\\
&=& u\cdot ( (e\circ (a_1\cdots a_{n-1}))\cdot a_n)
+u\cdot ( (a_1\cdots a_{n-1})\circ( e\cdot a_n))
-u\cdot ( ((a_1\cdots a_{n-1})\circ e)\cdot a_n)\\
&=& ua_n\circ (a_1\cdots a_{n-1})
+ (ua_1\cdots a_{n-1})\circ a_n
- (ua_1\cdots a_{n-1}a_n)\circ e\\
&=&\sum_{1\leq i\leq n-1}(u a_1\cdots \hat{a_i}\cdots a_n)\circ a_i -(n-1)(u a_1\cdots a_n)\circ e + (ua_1\cdots a_{n-1})\circ a_n\\
&=&\sum_{1\leq i\leq n}(u a_1\cdots \hat{a_i}\cdots a_n)\circ a_i  -(n-1)(u a_1\cdots a_n)\circ e.
\end{eqnarray*}
\end{proof}

\begin{lemma}\label{root max}
For any $T=[u_{1,2}\circ u_{1,1}\circ u_{2,1}\circ\cdots \circ u_{n,1}]_{_L}\in GDNP(X)$ where $u_{1,2}$, $u_{1,1}$, $\dots $, $u_{n,1}$ $\in [X], n\geq 1$,  we have
$T=\sum\alpha_lT_l$
where each $T_l$ is a \npt\ with $r(T_l)=r(T)=|T_l|_{\circ}=|T|_{\circ}, \ |T_l|_{_X} =|T|_{_X} $.
\end{lemma}
\begin{proof}
We may assume that $u_{1,1} \geq \cdots\geq  u_{n,1}$ by the deg-lex
order on $[X]$ and $u_{1,2}=b_1\cdots b_m,\  b_1\geq \cdots \geq b_m, \ b_1,\dots , b_m \in X, \ m\geq 0$.

Let $t=max\{i \mid |u_{i,1}|_{_X}>1, 1\leq i\leq n \}$. Induction on $t$.
If $t=0$ and $ m\leq 1$, then $T$ is a \npt.
If  $t=0$ and $  m>1$, then $T=b_2\cdots b_m[b_1\circ u_{1,1}\circ u_{2,1}\circ\cdots \circ u_{n,1}]_{_L}$ is a \npt.

If $t>1$ and $  u_{1,1}=a_1\cdots a_q\in [X]$, then
$T=[u_{1,2}\circ (a_1\cdots a_q)\circ u_{2,1}\circ\cdots \circ u_{n,1}]_{_L}
=\sum_{1\leq i\leq q}[(u_{1,2} a_1\cdots \hat{a_i}\cdots a_q)\circ a_i \circ u_{2,1}\circ\cdots \circ u_{n,1}]_{_L}
 -(q-1)[(u_{1,2} a_1\cdots a_n)\circ e\circ u_{2,1}\circ\cdots \circ u_{n,1}]_{_L}$.
By induction, the result follows.
\end{proof}

\begin{lemma}\label{generating set}
The set of  the \nptx\ over $X$ forms a linear generating set
 of  $GDNP(X)$.
\end{lemma}

\begin{proof}
Induction on $|T|_{\circ}$. If $|T|_{\circ}=0$, then it is obvious.
Suppose that the result holds for any $T$ with $|T|_{\circ}<t$.
 Let $|T|_{\circ}=t\geq 1$.
 Induction on $|T|_{\circ}-r(T)$.

 If $|T|_{\circ}-r(T)=0$, then by Lemmas \ref{prop n-1} and
\ref{root max}, the result follows.

Suppose that the result holds for any $T$ with $|T|_{\circ}<t$, or $|T|_{\circ}=t,\ |T|_{\circ}-r(T)< p$. Let
$|T|_{\circ}=t,\ |T|_{\circ}-r(T)=p$. By Lemma \ref{Novikov tableaux}, we may assume that
$
T=[u_{1,r_1+1}\circ A_1 \circ A_{2}\circ \cdots \circ  A_{n}]_{_L}
$, where
$
A_{i}=[u_{i,  r_i}\circ \cdots \circ  u_{i,  2}\circ u_{i,  1}]_{_R}$   for all $ 1\leq i \leq n.
$

Suppose $u_{i,1}=c_1\cdots c_q,\  q\geq 2$ for some $1\leq i\leq n$. If $r_i=1$, then
 by Lemma \ref{root 1},  we have
\begin{eqnarray*}
T&=& [u_{1,r_1+1}\circ A_1 \circ A_{2}\circ \cdots \circ  A_{n}]_{_L}=[u_{1,r_1+1}\circ A_i \circ A_{1}\circ \cdots \circ \hat{ A_{i}}\circ \cdots \circ  A_{n}]_{_L}\\
&=&\sum_{1\leq l\leq q}[(u_{1,r_1+1}\cdot c_1\cdots \hat{c_l}\cdots c_q)\circ c_l )]_{_R}\circ A_1\circ \cdots \circ \hat{ A_{i}}\circ \cdots \circ  A_{n}]_{_L}\\
&&-(q-1)[(u_{1,r_1+1}c_1\cdots c_q)\circ e )]_{_R}\circ A_1\circ \cdots \circ \hat{ A_{i}}\circ \cdots \circ  A_{n}]_{_L}.
\end{eqnarray*}

If $r_i>1$, then
 by Lemma \ref{root 1},  we have
\begin{eqnarray*}
T&=& [u_{1,r_1+1}\circ A_1 \circ A_{2}\circ \cdots \circ  A_{n}]_{_L}=[u_{1,r_1+1}\circ A_i \circ A_{1}\circ \cdots \circ \hat{ A_{i}}\circ \cdots \circ  A_{n}]_{_L}\\
&=&\sum_{1\leq l\leq q}[u_{1,r_1+1}\circ [u_{i,r_i}\circ \cdots \circ u_{i,3}\circ (u_{i,2}\cdot c_1\cdots \hat{c_l}\cdots c_q)\circ c_l )]_{_R}\circ A_1\circ \cdots \circ \hat{ A_{i}}\circ \cdots \circ  A_{n}]_{_L}\\
&&-(q-1)[u_{1,r_1+1}\circ [u_{i,r_i}\circ \cdots \circ u_{i,3}\circ (u_{i,2}c_1\cdots c_q)\circ e )]_{_R}\circ A_1\circ \cdots \circ \hat{ A_{i}}\circ \cdots \circ  A_{n}]_{_L}.
\end{eqnarray*}
Repeating this process if there is some other $j$ such that $|u_{j,1}|_{_X}>1$. So by induction hypothesis and Lemma \ref{interchange two row leaves},  we may assume that $|u_{i,1}|_{_X}\leq 1$ for any $1\leq i\leq n$.

Induction on $|T|_{_X}$. If $|T|_{_X}=0$, then $T$ is a \npt\ over $\{e\}$.
Suppose that the result holds for any $T$ with $|T|_{\circ}<t$, or $|T|_{\circ}=t,\ |T|_{\circ}-r(T)< p$, or
$|T|_{\circ}=t,\ |T|_{\circ}-r(T)=p, |T|_{_X}<q$. Let $|T|_{\circ}=t, \ |T|_{\circ}-r(T)=p,\ |T|_{_X}=q$.

If $|u_{1,r_1+1}|_{_X}\leq 1$, then $T$ is already a \npt.
If $|u_{1,r_1+1}|_{_X}>1$, say $u_{1,r_1+1}=b_1\cdots b_m$,  then
$
T=b_2\cdots b_{m} \cdot [b_1 \circ A_1 \circ A_{2}\circ \cdots \circ  A_{n}]_{_L}
$.
By induction and Lemma \ref{interchange two row leaves}, the result follows.
\end{proof}

\begin{theorem}\label{linear independent}
The set of  the \nptx \ over $X$ forms a linear basis of $GDNP(X)$.
\end{theorem}

\begin{proof}
By Lemma \ref{generating set}, it is enough to show that the set of the  \nptx\ over $X$ is  linear independent. By Lemma \ref{become novikov poisson},
$(C(X), \cdot, \circ)$ is a \npa. Now we  show that the \npa\  homomorphism
$\varphi:  \ (GDNP(X),\cdot, \circ) \longrightarrow(C(X), \cdot, \circ)$
induced by $\varphi(a)=a, a\in X$, is injective.
Given  any \npt\
$
T=b_1\cdots b_m[a_{1,r_1+1}\circ A_1 \circ \cdots \circ  A_{n}]_{_L},
$
where
$
 A_{i}=[a_{i,  r_i}\circ \cdots \circ  a_{i,  2}\circ a_{i,  1}]_{_R},    \  1\leq i \leq n,
$
by Lemma \ref{keep D}, we have
$\overline{\varphi(T)}=D^{r_{1}}a_{1,1}\ast D^{r_{2}}a_{2,1}\ast \cdots \ast D^{r_{n}}a_{n,1} \ast  a_{1,r_1+1}\ast\cdots
 \ast   a_{1,2}\ast a_{2,r_2}\ast\cdots \ast   a_{2,2}\ast \cdots \ast  a_{n,r_n} \ast \cdots \ast  a_{n,2}\cdot b_1\cdots b_m $.
 Suppose that $\sum_{1\leq i\leq t}\alpha_iT_i=0,$ where each $\alpha_i\in k, \ T_i $ is a \npt\ and
$T_i\neq T_j  $ for any $i\neq j, \ 1\leq i,j\leq t$.
 Then $\sum_{1\leq i\leq t}\alpha_i\varphi(T_i)=0$.
 Since $T_i\neq T_j \Rightarrow $
$\overline{\varphi(T_i)}\neq \overline{\varphi(T_j)}$, we have
$\alpha_i=0$ for any $1\leq i\leq t$.
\end{proof}

\section{A Poincar\'{e}-Birkhoff-Witt theorem}\label{section pbw np}
For any $S\subseteq C(X)$,  let us denote by $Id[S]$ the ideal of $C(X)$ generated by $S$ and
$
C(X|S):= C(X)/ Id[S]
$
the \cca\ generated by $X$ with defining relations $S$.
Noting that $w\ast D^{t}s=w\ast  D^{t}s\cdot e=w\ast e  \cdot  D^{t}s$,   we have
$$
Id[S]=span_{k}\{w \cdot D^{t}s\mid w\in C[X],\   j,t\in \mathbb{Z}_{\geq 0}, \  s\in S\}.
$$
By Lemma \ref{become novikov poisson}, $(C(X), \cdot, \circ)$
becomes a \npa\ with respect to
$
f\circ g=f \ast Dg,   \mbox{ for any }   f,  g\in C(X).
$

Recall that
$
  GDNP_{_0}(X)= span_{k}\{w\in C[X] \mid wt(w)=0\}
$
and $GDNP_{_0}(X)$ is  a  subalgebra of
 $(C(X), \cdot,  \circ)$ (as \npa).

\begin{lemma}\label{isomorphism}
Let $\varphi  :(GDNP(X),\cdot, \circ) \longrightarrow (GDNP_{_0}(X),\cdot, \circ)$ be the \npa\
 homomorphism induced by $\varphi(a)=a,\  a\in X$. Then $\varphi$ is an isomorphism.
\end{lemma}
\begin{proof}
 It is clear that $\varphi(GDNP(X))\subseteq GDNP_{_0}(X)$. To show that
 $GDNP(X)\cong GDNP_{_0}(X)$, we only need to show that for any word  $w\in C[X]$ with $wt(w)=0$,
 we have $w\in \varphi(GDNP(X))$. Since $wt(w)=0$, we have
$$
w=D^{r_1} a_{1,1}\ast\cdots \ast D^{r_n}a_{n,1}\ast
 c\ast a_{1,r_1}\ast \cdots \ast a_{1,2}\ast\cdots \ast a_{n,r_n}\ast\cdots\ast a_{n,2}\cdot b_1\cdots b_m \in C[X]
$$
 for some $n\geq 0, m\geq 0, \ r_i\geq 1$,
 $ c, a_{i,j}\in X\cup\{e\},\  b_l\in X, \ 1\leq i\leq n,\  1\leq j\leq r_i,\ 1\leq l\leq m$.
Let $T=b_1\cdots b_m[c\circ A_1 \circ A_2 \circ \cdots\circ A_n]_{_L}$, where
$A_{i}=[a_{i,  r_i}\circ \cdots \circ a_{i,  2}\circ a_{i,  1}]_{_R},   \  \  1\leq i \leq n$. Then $\overline{\varphi(T)}=w $.

Induction on $w$. If $w=e$, then
$w=\varphi(e)\in \varphi(GDNP(X))$.
Suppose that  the result holds for any word $<w$.
Since $\overline{\varphi(T)}=w $, by induction we have
 $w-\varphi(T)\in \varphi(GDNP(X))$. Thus $w\in \varphi(GDNP(X))$.
\end{proof}

\begin{lemma}\label{Novikov ideal}
Let $S$ be a nonempty subset of $GDNP_{_0}(X)$ and $Id(S)$ be the  ideal of $(GDNP_{_0}(X),\cdot,\circ)$ generated by $S$. Then
$$
Id(S)=span_{k}\{w\cdot D^{t}s \mid w\in C[X], \  t\in \mathbb{Z}_{\geq 0},  \  s\in S,  \  wt(\overline{w\cdot D^{t}s})=0\}.
$$
\end{lemma}
\begin{proof}
It is clear that the right part is a \npa\ ideal that contains $S$. We just need to show that $w\cdot D^{t}s \in Id(S)$
whenever $wt(\overline{w\cdot D^{t}s})=0$. Since $wt(\overline{w\cdot D^{t}s})=0$,
we have
$$ w\cdot D^{t}s=D^{t}s\ast d_{1}\ast \cdots\ast  d_{t}\ast D^{r_1}c_{1}\ast \cdots\ast  D^{r_n}c_{n}\ast a_{1}\ast \cdots\ast  a_{m}\cdot b_1\cdots b_l,
$$
where
$w=D^{r_1}c_{1}\ast \cdots\ast  D^{r_n}c_{n}\ast a_{1}\ast \cdots\ast  a_{m}\ast d_{1}\ast \cdots\ast  d_{t}\ast b_1\cdot b_2\cdots b_l, n\geq 0, m\geq 0, t\geq 0$,
$\ m=r_1+\cdots  +r_n-n$ and $ r_n\geq r_{n-1}\geq \dots  \geq r_1\geq 1.$
So the lemma will be clear if we show

(i)   $D^{t}s\ast d_{1}\ast \cdots\ast  d_{t}\in Id(S)$ whenever $s\in S$;

 (ii)   $ f \ast D^{r}c\ast a_{1}\ast \cdots\ast  a_{r-1}\in Id(S)$ whenever $f\in Id(S)$.

 To prove (i),    we use induction on $t$. If $t=0,  $ it is clear.
 Suppose that it holds for any $t\leq p$. Then
 \begin{eqnarray*}
&&D^{p+1}s\ast d_{1}\ast \cdots\ast  d_{p+1}\\
&=&d_{p+1}\ast D^{p+1}s\ast d_{1}\ast \cdots\ast d_{p}\\
&= &d_{p+1}\circ(D^{p}s\ast d_{1}\ast \cdots\ast   d_{p})
-d_{p+1}\ast \sum_{1\leq i \leq p}D^{p}s\ast d_{1}\ast \cdots\ast   Dd_{i}\ast \cdots\ast  d_{p}\\
&=&  d_{p+1}\circ(D^{p}s\ast d_{1}\ast \cdots\ast   d_{p})
-\sum_{1\leq i \leq p} (D^{p}s\ast d_{1}\ast \cdots\ast   d_{i-1}\ast d_{i+1}\cdots  d_{p+1})\circ d_{i}\in  Id(S).
 \end{eqnarray*}
To prove (ii),   we use induction on $r$.
If $r=1,  $ it is clear.
 Suppose that it holds for any $r\leq p$. Then
  \begin{eqnarray*}
&&f \ast D^{p+1}c\ast a_{1} \ast\cdots   \ast a_{p}\\
&=&f\circ (D^{p}c\ast a_{1} \ast\cdots   \ast a_{p})
- f \ast\sum_{1\leq i \leq p}D^{p}c\ast a_{1} \ast\cdots  \ast   Da_{i}  \ast \cdots \ast a_{p}\\
&=&f\circ (D^{p}c\ast a_{1} \ast\cdots   \ast a_{p})
- \sum_{1\leq i \leq p}(f \ast D^{p}c\ast a_{1} \ast\cdots  \ast   a_{i-1}  \ast   a_{i+1} \ast \cdots \ast a_{p})\circ a_i
\in  Id(S).
 \end{eqnarray*}
So $Id(S)=span_{k}\{w\cdot D^{t}s \mid w\in C[X], \  t\in \mathbb{Z}_{\geq 0},  \  s\in S,  \ wt(\overline{w\cdot D^{t}s})=0\}$.
\end{proof}

The following theorem is a Poincar\'{e}-Birkhoff-Witt theorem for \npas.

\begin{theorem}
Any \npa\ $GDNP(X|S)$ can be  embedded into
its universal enveloping \cca\ $C(X|S)$.
\end{theorem}
\begin{proof}
By Lemmas \ref{isomorphism} and \ref{Novikov ideal}, we get $GDNP(X)\cong GDNP_{_0}(X)$ and
 $$
 \frac{GDNP(X)}{Id(S)}\cong \frac{GDNP_{_0}(X)}{Id(\varphi S)}=\frac{GDNP_{_0}(X)}{GDNP_{_0}(X)\cap Id[\varphi S]}\cong \frac{GDNP_{_0}(X)+Id[\varphi S]}{Id[\varphi S]}
 \leq C(X|\varphi S) ,
 $$
 where $\varphi $ is defined in Lemma \ref{isomorphism}.
\end{proof}

\section{Differential GDN-Poisson algebras}

\begin{definition}  \emph{A} differential GDN-Poisson algebra \emph{$(\mathcal{A}, \cdot, \circ)$ is a  \npa\  satisfying the following identity}:
$$
x\circ (y\cdot z)=(x\circ y)\cdot z+(x\circ z)\cdot y \ \ \ \ \ \ \ \ \ (\lozenge)
$$
\end{definition}

Let $D GDNP(X)$  denote  a free differential GDN-Poisson algebra generated by a
well-ordered set $X$.  It is clear that
$$
D GDNP(X)=GDNP(X|\lozenge ):= \frac{ GDNP(X)}{ Id(\lozenge )},
$$
where $Id(\lozenge )$ is the ideal of $ GDNP(X)$ generated by $(\lozenge )$.
So for any \npa\ $\mathcal{A}$ generated by $X$, if $\mathcal{A}$ satisfies $(\lozenge )$, then
$\mathcal{A}\cong DGDNP(X|R)$ for some $R\subseteq DGDNP(X)$.
For any $x\in D GDNP(X)$, since  $x\circ  e=x\circ( e\cdot e)=x\circ  e+x\circ e$, we have $x\circ  e=0$.

For any word $T$ in  $F(\Omega, X\cup \{e\})$,  the number of $ e$ that appears in $T$ will be denoted by $|T|_{ e}$.
For any word $T\in D GDNP(X)$, we consider $T$ as a word in
$F(\Omega, X\cup \{e\})$. Then   $|T|_{e}$ also makes sense.
For example, let $a,b\in X, T=(((( e\cdot e)\circ a)\circ (a\circ (b\circ e)))\cdot e)\cdot e$. Then $|T|_{_X}=3,\  |T|_{ e}=5.$

Define
$$
[ e\circ  e\circ \cdots\circ  a]_i:= [\underbrace{ e\circ  e\circ \cdots\circ   e}_{i \mbox{ times }}\circ a]_{_R},\ i\geq 0, \ a\in X.
$$
We call $T$ a normal word if
$$
T=[ e\circ  e\circ \cdots \circ a_1]_{i_1}\cdot [ e\circ  e\circ \cdots \circ a_2]_{i_2}\cdots [ e\circ  e\circ \cdots \circ a_n]_{i_{n}}\  (T=e \mbox{ if } n=0),
$$
where $a_1, \dots , a_n\in X,\ (i_1,a_1)\geq\cdots \geq (i_n,a_n), n\in \mathbb{Z}_{\geq 0}$.

\begin{lemma}\label{1 normal}
For any normal word $T\in D GDNP(X)$, we have
$ e\circ T=\sum_{i}\alpha_{i}T_i$, where each $\alpha_i\in k$,  $T_i$ is normal with $|T_i|_{_X}=|T|_{_X}$.
\end{lemma}
\begin{proof}
Let $T=[ e\circ  e\circ \cdots \circ a_1]_{i_1}\cdots [ e\circ  e\circ \cdots \circ a_n]_{i_{n}}, n\in \mathbb{Z}_{\geq 0}$. Induction on $n$. If $n=1$, the result follows.
Let $n=t\geq 2$.
Then
$
 e\circ T=( e\circ[ e\circ  e\circ \cdots \circ a_1]_{i_1})\cdot [ e\circ e\circ \cdots \circ a_2]_{i_2}\cdots [ e\circ  e\circ \cdots \circ a_t]_{i_{t}}
+( e\circ([ e\circ  e\circ \cdots \circ a_2]_{i_2}\cdots [ e\circ  e\circ \cdots \circ a_t]_{i_{t}}))\cdot [ e\circ  e\circ \cdots \circ a_1]_{i_1}.
$
The result follows by induction.
\end{proof}
\begin{lemma}
For any word $T\in D GDNP(X)$, we have
$ T=\sum_{i}\alpha_{i}T_i$, where each $\alpha_i\in k$ and $T_i$ is normal.
\end{lemma}
\begin{proof}
Induction on $|T|_{_X}$. If $|T|_{_X}=0$, then it is clear that $T=0 $ or $ e$.
 Let $|T|_{_X}=n\geq 1$.
Now, induction on $|T|_{ e}$.

If  $|T|_{ e}=0$, then $T=T_1\cdot T_2$ or
$T=T_1\circ T_2=(T_1\cdot  e)\circ T_2=T_1\cdot ( e\circ T_2)$,  where $|T_1|_{_X}<n$, $| e\circ T_2|_{_X}=|T_2|_{_X}<n$. By induction
hypothesis, the result follows.
Suppose that the result holds for any $T$
 with $|T|_{_X}<n$, or $|T|_{_X}=n, |T|_{ e}<t.$
Let $|T|_{_X}=n, |T|_{ e}=t\geq 1$,
 $T=T_1\cdot T_2  $  or $T=T_1\circ T_2=T_1\cdot ( e\circ T_2)$.

 If $|T_1|_{_X}<n$ and $| e\circ T_2|_{_X}=|T_2|_{_X}<n$,
then the result follows by induction.
If  $|T_1|_{_X}=0$ and $|T_2|_{_X}=n$, then $T_1=0$ or $ e$, and $|T_2|_{ e}<t$. By induction,
we may assume that $T_2$ is normal.  Then
$T_1\cdot T_2=0$ or $T_1\cdot T_2=T_2$; $T_1\circ T_2=0$ or $T_1\circ T_2= e\circ T_2$. By Lemma \ref{1 normal}, the result follows.
If  $|T_1|_{_X}=n$ and $ |T_2|_{_X}=0$, then $T_2=0$ or $ e$, and $|T_1|_{ e}<t$.  Then
$T_1\cdot T_2=0$ or $T_1\cdot T_2=T_1$; $T_1\circ T_2=0$.  By induction,
the result follows.
\end{proof}
Let $k\{X\}$ be a free commutative associative differential algebra generated by a
well-ordered set $X$ ($De=0$) with unit  $ e$ and one linear derivation $D$. Then it is clear that the set $[D^{\omega}X]$ of  a linear basis of $k\{X\}$ consists of
 $$
T=D^{r_1}a_1\cdots D^{r_n}a_n  \ \ \ ( T= e \mbox{ if } n=0),
$$
where $a_1, \dots, a_n\in X$, each
$r_i\geq 0$, $(r_1,a_1)\geq \cdots\geq (r_n,a_n)$.

Let $(\mathcal{A},\cdot, D)$ be a commutative associative differential algebra with one derivation $D$.
 Define the operation $\circ$ on $\mathcal{A}$ by
 $
x\circ y:= xDy, \  x,y \in \mathcal{A}.
$
Then $(\mathcal{A},\circ)$ forms a \npa\ (\cite{Xiaoping Xu Poisson}).
Moreover, for any $x,y,z\in \mathcal{A}$, we have
$x\circ(y z)=xD(yz)=x(Dy)z+xyDz=(x\circ y)z+(x\circ z) y.$

\begin{lemma}\label{DNP basis}
The set of  normal words forms a linear basis of $D GDNP(X)$.
In particular, $(D GDNP(X),\cdot,\circ)\cong (k\{X\},\cdot,\circ)$ as \npas,
where $\circ$ in $k\{X\}$ is defined as
$f\circ g=fDg $ for any $f,g\in k\{X\}$, and  $(D GDNP(X), \cdot, \partial)\cong (k\{X\}, \cdot, D)$ as
 commutative associative differential algebras,
 where $\partial$ in $D GDNP(X)$ is defined as $\partial(f)= e\circ f $ for any $f\in D GDNP(X)$.
\end{lemma}

\begin{proof}
Let $\theta:D GDNP(X)\longrightarrow k\{X\}$ be a \npa\ homomorphism induced by $\theta(a)=a,  a\in X$.
Then  it is clear that $\theta$ is an isomorphism.
So the set $\theta^{-1}([D^{\omega}X])$ of  normal words forms a linear basis of $D GDNP(X)$.

It is clear that  $(D GDNP(X), \cdot , \partial)$ forms a commutative  associative differential algebra. Moreover,
$\theta(\partial(u v))
=\theta( e\circ(u v))= \theta e\circ\theta(u v)
=e\circ(u v)=D(uv)=D(u)v+uD(v)=\theta(( e\circ u) v)+\theta(( e\circ v) u)
=\theta((\partial u)v+(\partial v)u)$,
so $\theta$ is a commutative  associative differential algebra isomorphism.
\end{proof}

\begin{theorem}\label{NP differential iso}
Let $\theta:D GDNP(X)\longrightarrow k\{X\}$ be a \npa\ homomorphism induced by $\theta(a)=a,  a\in X$.
Then any \npa\ $GDNP(X|S)$ satisfying $(\lozenge )$ is isomorphic to $k\{X|\theta(S)\}$ both
as  \npas\ and as commutative associative  differential algebras.
\end{theorem}
\begin{proof}
Let $Id(\theta (S))$ be the ideal of  $(k\{X\}, \cdot,\circ)$ as \npa\ and $Id[\theta (S)]$ the
ideal of $(k\{X\},\cdot, D )$  as commutative associative  differential algebra. It is clear that
$Id[\theta (S)]=\{\sum\alpha_iu_i D^{r_i} (\theta s_i)\mid \alpha_i\in k, u_i \in k\{X\}, s_i\in S \}.$
Since $ D^{r}g= [\underbrace{ e\circ  e\circ \cdots\circ   e}_{r \mbox{ times }}\circ g]_{_R}$ for any $g \in k\{X\}$, it is straightforward to show that  $Id(\theta (S))=Id[\theta (S)]$.
 By Lemma \ref{DNP basis}, we have
 $$
 GDNP(X|S)= \frac{ GDNP(X)}{Id(S)}\cong \frac{\theta(D GDNP(X))}{Id(\theta(S))}=\frac{k\{X\}}{Id[\theta (S)]}= k\{X| \theta(S)\}.
 $$
\end{proof}
\begin{corollary}
$(\lozenge )$ holds in a \npa\ $\mathcal{A}$ if and only if $\mathcal{A}$ is embeddable into a commutative associative differential
algebra $\mathcal{B}$ with one derivation $D$, where
$\circ$ is defined as $f\circ g=fDg$ for any $f,g\in \mathcal{B}$.
\end{corollary}\
\begin{proof}
If $(\lozenge )$ holds in $\mathcal{A}$, then by Theorem \ref{NP differential iso}, $\mathcal{A}$ can be embedded into a commutative associative differential
algebra. Conversely, assume that $\theta: \mathcal{A}\longrightarrow \mathcal{B}$ is an embedding. Then
for any $x,y,z\in \mathcal{A}$, we have
$\theta( x\circ (y z))=\theta(x)\circ (\theta(y) \theta(z))=\theta(x)D(\theta(y)\theta(z))=(\theta(x)D\theta(y)) \theta(z)+(\theta(x)D\theta(z)) \theta(y)=(\theta(x)\circ\theta(y)) \theta(z)+(\theta(x)\circ \theta(z))\theta(y)=\theta ((x\circ y) z+(x\circ z) y)$.  Since $\theta$ is injective,  $(\lozenge )$ holds in $\mathcal{A}$.
\end{proof}
\section*{Acknowledgment} We would like to thank the  referee for  valuable suggestions.

\end{document}